\documentclass[12pt]{amsart}
\usepackage{amssymb,amsmath}
\usepackage{graphicx}

\def\Fig1{\begin{figure}[htb]\begin{center}
\includegraphics[width=6.5cm]{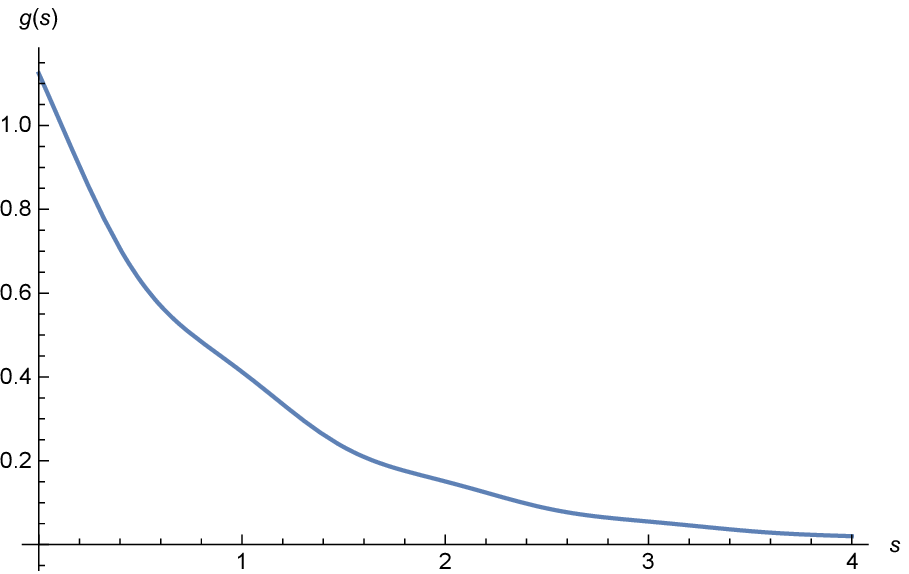}\qquad\,\,\includegraphics[width=6.5cm]{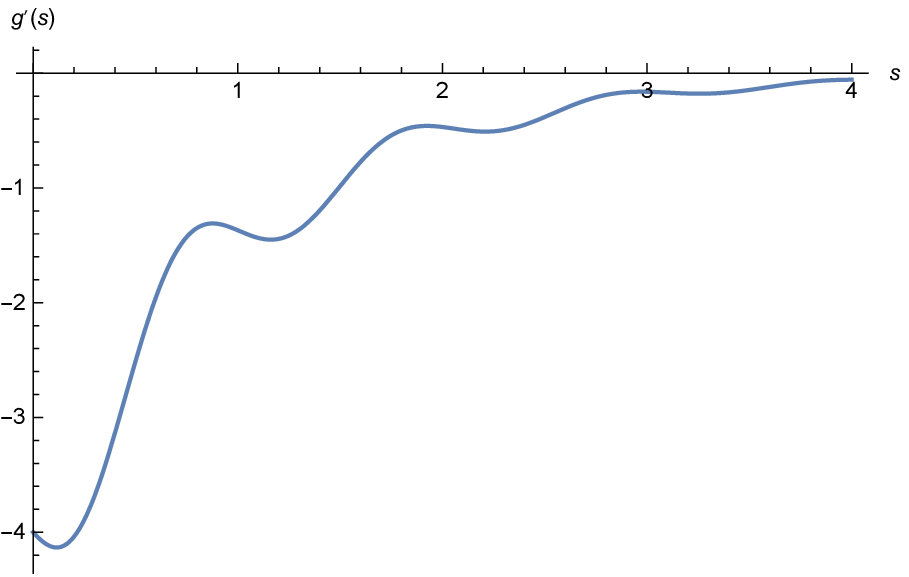}\\
{\tiny fig.\ $\!$1$\,\,$ Portraits of $g(s)$ and $g'(s)$}
\end{center}
\end{figure}}

\newtheorem{proposition}{Proposition}[section]
\newtheorem{theorem}[proposition]{Theorem}

\newtheorem{lemma}[proposition]{Lemma}
\theoremstyle{definition}
\newtheorem{definition}[proposition]{Definition}
\newtheorem{remark}[proposition]{Remark}
\newtheorem{example}[proposition]{Example}
\numberwithin{equation}{section}

\def\R {\mathbb{R}}

\def\e{{\rm e}}
\def\H {{\mathcal H}}
\def\q {{\mathsf q}}
\def\Q {{\mathsf Q}}
\def\C {{\mathcal C}}
\def\D {{\mathfrak D}}
\def\E {{\mathsf E}}
\def\F {{\mathsf F}}
\def\T {{\mathbb T}}
\def\I {{\mathbb I}}
\def\S {{\mathbb S}}

\def\M {{\mathcal M}}
\def \l {\langle}
\def \r {\rangle}
\def \au {\rm}
\def \ti {\it}
\def \jou {\rm}
\def \bk {\it}
\def \no#1#2#3 {{\bf #1} (#3), #2.}
\def \eds#1#2#3 {#1, #2, #3.}
\begin{document}

%%%%%%%%%%%%%%%%%%%%%%%%%%%%%%%%%%%%%%%%%%%%
\title[MGT equation with memory with nonconvex kernels]
{On the Moore-Gibson-Thompson equation\\ with memory with nonconvex kernels}

\author[M. Conti, L. Liverani and V. Pata]{Monica Conti, Lorenzo Liverani and Vittorino Pata}

\address{Politecnico di Milano - Dipartimento di Matematica
\newline\indent
Via Bonardi 9, 20133 Milano, Italy}
\email{monica.conti@polimi.it {\rm (M. Conti)}}
\email{lorenzo.liverani@polimi.it {\rm (L. Liverani)}}
\email{vittorino.pata@polimi.it {\rm (V. Pata)}}

%%%%%%%%%%%%%%%%%%%%%%%%%%%%%%%%%%%%%%%%%%%%

%%%%%%%%%%%%%%%%%%%%%%%%%%%%%%%%%%%%%%%%%%%%
\begin{abstract}
We consider the MGT equation with memory
$$\partial_{ttt} u + \alpha \partial_{tt} u - \beta \Delta \partial_{t} u
- \gamma\Delta u + \int_{0}^{t}g(s) \Delta u(t-s) ds = 0.$$
We prove an existence and uniqueness result removing the convexity assumption on
the convolution kernel $g$,
usually adopted in the literature.
In the subcritical case $\alpha\beta>\gamma$,
we establish the exponential decay of the energy, without leaning on
the classical differential
inequality involving $g$ and its derivative $g'$, namely,
$$g'+\delta g\leq 0,\quad\delta>0,$$
but only asking that $g$ vanishes exponentially fast.
\end{abstract}
%%%%%%%%%%%%%%%%%%%%%%%%%%%%%%%%%%%%%%%%%%%%

\maketitle

%%%%%%%%%%%%%%%%%%%%%%%%%%%%%%%%%%%%%%%%%%%%
\section{Preamble: The MGT Equation}

\noindent
Given a smooth bounded domain $\Omega\subset\R^3$,
the Moore-Gibson-Thompson (MGT) equation
is the third-order PDE
\begin{equation}
\label{MGT}
\partial_{ttt} u + \alpha \partial_{tt} u-\beta \Delta \partial_t u - \gamma \Delta u =0,
\end{equation}
ruling the evolution of the unknown variable
$u=u(\boldsymbol{x},t):\Omega\times [0,\infty) \to \R$.
Here, $-\Delta$ is the Laplace-Dirichlet operator, while
$\alpha,\beta,\gamma>0$ are fixed structural parameters.

\smallskip From the physical viewpoint,
\eqref{MGT} is a wave-type equation arising
in the context of acoustic wave propagation
with the so-called second sound, being $u$
the acoustic pressure,
where the paradox of the infinite speed of propagation
is eliminated by replacing the Fourier law by the Maxwell-Cattaneo one. Such a feature
explains the presence of the third-order derivative in time.
The model equation first appeared in a very old paper of Stokes \cite{STO},
but it has received a considerable attention only in recent years, where a number of authors devised
many possible applications in nonlinear acoustic and in thermal relaxation in viscous
gases and fluids (see, e.g., \cite{PJ,KAL,MG,TOM}).

\smallskip
The mathematical analysis of \eqref{MGT} is very intriguing, and it has been carried out
in several papers (see \cite{BU1,BU2,DPMGT,KL,KLM,KLP,KN,LAS,MDT}).
For all values of the positive parameters
$\alpha,\beta,\gamma$, the (linear) equation turns out to generate a strongly continuous semigroup of solutions
on the phase space
$$\H= H_0^1(\Omega)\times  H_0^1(\Omega)\times L^2(\Omega).$$
Here, $L^2(\Omega)$ is the Lebesgue space of square summable functions on $\Omega$, whereas $H_0^1(\Omega)$
is the Sobolev space of square summable functions, along with their derivatives, with null trace on the
boundary $\partial \Omega$. However, the asymptotic behavior of the solutions is strongly influenced by the
\emph{stability number}
$$\varkappa = \beta - \frac{\gamma }{\alpha}.$$
Indeed, the picture is the following:
\begin{itemize}
\item[$\bullet$] If $\varkappa>0$ the solutions decay exponentially fast.
\smallskip
\item[$\bullet$] If $\varkappa=0$ the energy of the system is conserved.
\smallskip
\item[$\bullet$] If $\varkappa<0$ there are solutions with an exponential blow up.
\end{itemize}
For this reason, the regimes $\varkappa>0$, $\varkappa=0$ and $\varkappa<0$ are usually referred to
in the literature as
\emph{subcritical}, \emph{critical} and \emph{supercritical}, respectively.
Quite interestingly, in the subcritical case $\varkappa>0$, equation~\eqref{MGT} serves also
as a model for the description of linearly viscoelastic solids (see \cite{DDR,DPMGT}).
%%%%%%%%%%%%%%%%%%%%%%%%%%%%%%%%%%%%%%%%%%%%

%%%%%%%%%%%%%%%%%%%%%%%%%%%%%%%%%%%%%%%%%%%%%%%%%
\section{Introduction}

\noindent
Lately, the following integral relaxation
of the MGT equation~\eqref{MGT} has been proposed:
$$
\partial_{ttt} u(t) + \alpha \partial_{tt} u(t) - \beta \Delta \partial_t u(t) -
\gamma\Delta u(t) + \int_{0}^{t}g(s) \Delta w(t-s) ds = 0,
$$
where $w$ is a linear combination of $u$ and $\partial_t u$,
and $g$ is a positive decreasing kernel.
The convolution term
has been added in order to take into account molecular relaxation phenomena,
which introduce a sort of delay in the dynamics, producing nonlocal effects in time
\cite{LEB,OST}.
Loosely speaking, the outcome is that at any time $t$ the system keeps memory of
(and it is influenced by) its past evolution.
Depending on the physical properties of the medium and the environment,
the variable $w(t)$ can fall into one of the following types:
$$
w(t) =
\begin{cases}
u(t), \\
\partial_t u(t),\\
\lambda u(t) + \partial_t u(t),\quad \lambda>0.
\end{cases}
$$
The three situations above are usually named memory of type I, type II and type III, respectively.
For the cases of memory of type II and III, we address the reader to the works \cite{ACJR,DLPII,LW1}.

In the present paper, we are interested to the case of memory of type~I, which is in a sense the most natural,
and certainly the one that has received greater attention.
Accordingly, our equation reads
\begin{equation}
\label{volDelta}
\partial_{ttt} u(t) + \alpha \partial_{tt} u(t) - \beta \Delta \partial_t u(t) -
\gamma\Delta u(t) + \int_{0}^{t}g(s) \Delta u(t-s) ds = 0.
\end{equation}
The two main issues concerning \eqref{volDelta} are:
\begin{enumerate}
\item[$\bullet$] Existence and uniqueness of weak solutions in the phase space $\H$.
\smallskip
\item[$\bullet$] Uniform decay properties of the associated natural energy, given by\footnote{Here,
$\|\cdot\|$ denotes the usual norm in $L^2(\Omega)$ or in $[L^2(\Omega)]^3$.}
$$
\E(t)=\|\nabla u(t)\|^2+\|\nabla\partial_t u(t)\|^2
+\|\partial_{tt}u(t) \|^2 + \int_0^t -g'(s) \|\nabla u(t)-\nabla u(t-s)\|^2 ds.
$$
\end{enumerate}
Evidently, suitable hypotheses on the kernel $g$ have to be made.
The basic assumptions on $g$, employed with no exception in the literature,
are:
$g\in\C^1(\R^+)$ with $g\geq 0$ and $g'\leq 0$, both summable. Besides,
the total mass of $g$ is required to be strictly less than $\gamma$.

Within the further assumption that $g$ is convex, existence and uniqueness of solutions
is proved in \cite{LW1} for the subcritical case $\varkappa> 0$, and in
\cite{ACJR} for the critical case $\varkappa= 0$ (but asking also~\eqref{Dafermos}
below).
Actually, by means of a ``pumping" technique introduced in \cite{DPMGT},
it is not hard to extend the result
for all negative values of the stability number $\varkappa$, hence covering
the supercritical case as well.
In fact, the convexity of $g$ plays a crucial role in the well-posedness
analysis.
One exception is the paper \cite{LW2} (see also \cite{CNS1}), where the convexity of $g$ is not required, but it
is replaced by the rather restrictive condition
$$g(0)<\alpha^2\varkappa,$$
which clearly can take place only in the subcritical regime $\varkappa>0$.

As far as the uniform decay of the energy is concerned, here we are interested to analyze the
exponential case, although other decay patterns are possible (such as polynomial), depending on the form of $g$.
A first observation is that the memory term in \eqref{volDelta} introduces an extra dissipation
with respect to the one of the MGT equation~\eqref{MGT}.
Hence, in principle, one could expect~\eqref{volDelta} to exhibit stronger stability properties
than~\eqref{MGT}. For instance, one could think that the critical regime $\varkappa=0$ is driven into
subcritical when memory is added. This is utterly false
(unless  $-\Delta$ is replaced by a \emph{bounded} operator), even in the simplest possible case of the
exponential kernel. A counterexample in that direction is given in~\cite{DLP}.
Accordingly, in order to hope for exponential stability, we shall confine ourselves to the
subcritical case $\varkappa>0$.

The exponential decay at infinity of the energy $\E(t)$ is always proved within the classical assumption
\begin{equation}
\label{Dafermos}
g'(s)+\delta g(s) \leq 0,
\end{equation}
for some $\delta>0$ (see \cite{ACJR,LW2,LW1,CNS2,CNS1}), implying in
particular the exponential decay of the kernel
$$g(s)\leq g(0)\e^{-\delta s}.$$
The differential inequality~\eqref{Dafermos} is to some extent unsatisfactory, for it
imposes severe restrictions on the shape of the kernel, whereas the
exponential decay of $\E(t)$ should be reasonably expected if
$g$ is exponentially decaying solely, with no further information.

\smallskip
In light of the discussion above, the aim of the present paper is twofold:
\begin{enumerate}
\item[$\bullet$] Prove an existence and uniqueness result removing the convexity
assumption on $g$, and requiring in place a much weaker condition: roughly speaking,
the derivative $g'$ cannot be ``too small" where $g$ is concave down.
\smallskip
\item[$\bullet$] Prove the exponential decay of the energy in the subcritical regime $\varkappa>0$, assuming only that
$g$ is controlled from above by a negative exponential.
\end{enumerate}

\begin{remark}
In fact, by applying the trick devised in~\cite{DPMGT}, our proof of the existence and uniqueness
result can be immediately extended to cover all the values of $\varkappa$.
\end{remark}

We finally mention that the idea of studying the exponential decay of the energy
relaxing condition~\eqref{Dafermos} has been first devised
in the paper~\cite{CGP}, in the context of linear viscoelasticity. Nonetheless,
our techniques can be applied to obtain also decay types
other than exponential
(see e.g.\ \cite{LW2,CNS2}),
without assuming a differential inequality on $g$, but only requiring a proper control from above,
along the line of the recent work~\cite{Cippi}.

\subsection*{Plan of the paper}
In the next Section~\ref{SFSN} we stipulate the general assumptions,
defining the proper functional setting for our problem.
The main results are stated in Section~\ref{SSMR}.
The remaining of the paper is devoted to the proofs.
In Section~\ref{SRT} we study
the semigroup of right translations on a suitable memory space.
This object is exploited in Section~\ref{SAP},
in the analysis of a certain family of approximated problems, depending on a
smoothing parameter $\varrho>0$.
The well-posedness of the $\varrho$-problems, as well as their exponential stability,
are discussed in Sections~\ref{SAPEU}, \ref{SAPEI} and \ref{SAPEDE}.
In the final Section~\ref{SAPMR} we perform the limit $\varrho\to0$,
concluding the proofs of the main results.
%%%%%%%%%%%%%%%%%%%%%%%%%%%%%%%%%%%%%%%%%%%%

%%%%%%%%%%%%%%%%%%%%%%%%%%%%%%%%%%%%%%%%%%%%
\section{Functional Setting and Notation}
\label{SFSN}

\subsection*{Geometric spaces} Let $(H,\l\cdot,\cdot\r,\|\cdot\|)$ be a real Hilbert space,
and let $A:H\to H$ be a strictly positive selfadjoint operator
with domain $\D(A)\subset H$.
For $\sigma \in \R$, we define the hierarchy of (continuously) nested Hilbert spaces
$$H^\sigma = \D(A^{\sigma/2}),$$
endowed with the scalar products and norms
$$\l u,v\r_{\sigma}
=\l A^{\sigma/2}u,A^{\sigma/2}v\r\qquad\text{and}\qquad
\|u\|_\sigma=\|A^{\sigma/2}u\|.$$
The index $\sigma$ appearing  in the scalar products and norms will be always omitted whenever zero.
For $\sigma>0$, it is understood that $H^{-\sigma}$
denotes the completion of the domain, so that $H^{-\sigma}$ is the dual space of $H^\sigma$.
The symbol $\l\cdot,\cdot\r$ will also be used to denote the duality pairing between
$H^{-\sigma}$ and $H^\sigma$.
Finally, we introduce the phase space of our problem, namely, the product Hilbert space
$$\H = H^1 \times H^1 \times H,$$
with the usual product norm
$$\|(u,v,w)\|_{\H}^2=\|u\|_{1}^2+\|v\|_1^2+\|w\|^2.$$

\begin{remark}
For the concrete equation \eqref{volDelta} of the previous section, $A$ is the
Laplace-Dirichlet operator $-\Delta$ on
$H=L^2(\Omega)$, while $H^1=H_0^1(\Omega)$ and $H^2=H^2(\Omega)\cap H_0^1(\Omega)$.
\end{remark}

\subsection*{Properties of the memory kernel}
We assume that the memory kernel $g$ is a nonnull function
defined on $\R^+=(0,\infty)$ satisfying the following set of hypotheses:
\begin{enumerate}
\smallskip
\item[{\bf (g1)}] $g$ and its derivative $g'$ are summable and absolutely continuous on $\R^+$.
Besides,
$$\kappa=\int_{0}^{\infty}g(s) ds < \gamma.$$
\item[{\bf (g2)}] $g(s) \geq 0$ and $g'(s) \leq 0$ for every $s >0$.
\medskip
\item[{\bf (g3)}] $\exists\,\delta > 0$ such that for a.e.\ $s>0$
$$(\alpha - \delta)g'(s) - g''(s) \leq 0.$$
Without loss of generality we may (and do) assume $\delta<\alpha$.
\smallskip
\end{enumerate}

\begin{remark}
\label{Remg1}
Condition (g1) implies that
$$g(0)=\lim_{s\to 0}g(s)=\int_0^\infty -g'(s)ds <\infty.$$
\end{remark}

\begin{remark}
The derivative $g'$ can be unbounded about zero.
In addition, condition~(g3) along with the fact that $g'$ is summable yield
$$
\lim_{s \to \infty} g'(s) = 0.
$$
The proof of this limit is a nice real analysis exercise, and is left to the interested reader.
Besides, knowing that $g'$ is differentiable almost everywhere, (g3) is completely
equivalent to require
\begin{equation}
\label{gform}
-g'(s+t)\leq -g'(s)\e^{(\alpha-\delta)t},
\end{equation}
for all $s>0$ and $t\geq 0$. Indeed, one implication follows from the Gronwall lemma,
whereas the converse is proved by performing the limit of the quotient $\frac{g'(s+t)-g'(s)}{t}$
as $t\to 0^+$.
\end{remark}

\begin{example}
\label{exa}
Taking without loss of generality $\alpha-\delta=1$ and $\gamma$ large enough, the memory kernel
$$g(s)=\frac1{37}\e^{-s}\big[148+6\cos 6s +\sin 6s \big]$$
is easily seen to satisfy (g1)-(g3). Indeed,
$$g'(s)-g''(s)=-\e^{-s} \big [ 8+2\sin 6s - 6 \cos 6s \big]\leq 0.
$$
However, its derivative
$$g'(s)=-\e^{-s}\big [4+\sin 6s \big]$$
is oscillating all over $\R^+$. Accordingly, $g(s)$ is not convex, and not even ultimately convex
as $s\to\infty$.
\end{example}

\Fig1
%%%%%%%%%%%%%%%%%%%%%%%%%%%%%%%%%%%%%%%%%%%%

%%%%%%%%%%%%%%%%%%%%%%%%%%%%%%%%%%%%%%%%%%%%
\section{Statement of the Main Results}
\label{SSMR}

\noindent
Let (g1)-(g3) hold. In greater generality,
within the subcritical condition
\begin{equation}
\label{subcritical}
\varkappa = \beta - \frac{\gamma }{\alpha}>0,
\end{equation}
we will consider for $t>0$ the equation
\begin{equation}
\label{volterra}
\partial_{ttt} u(t) + \alpha \partial_{tt} u(t) + \beta A\partial_{t} u(t) + \gamma Au(t) - \int_{0}^{t}g(s) Au(t-s) ds = 0,
\end{equation}
subject to the initial conditions
\begin{equation}
\label{IC}
\begin{cases}
u(0) = u_0, \\
\partial_t u(0) = v_0, \\
\partial_{tt} u(0) = w_0,
\end{cases}
\end{equation}
where $(u_0,v_0,w_0)$ is an arbitrarily given vector of $\H$.

\begin{definition}
\label{solution}
A function $u$ is said to be a weak solution to problem \eqref{volterra}-\eqref{IC}
on the time interval $[0,T]$ if:
\begin{itemize}
\item[(i)] $(u,\partial_t u,\partial_{tt}u) \in \C([0,T],\H)$;
\smallskip
\item[(ii)] the initial conditions \eqref{IC} are satisfied; and
\smallskip
\item[(iii)] for every test function $\zeta\in H^1$ and a.e.\ $t\in[0,T]$,
$$\l \partial_{ttt} u(t), \zeta \r+ \alpha \l  \partial_{tt} u(t), \zeta \r + \beta \l  \partial_{t} u(t), \zeta \r_1
+ \gamma\l  u(t),\zeta \r_1 -\int_{0}^{t} g(s)\l u(t-s), \zeta\r_1 ds = 0.$$
\end{itemize}
We also define the corresponding energy by
\begin{equation}
\label{energy}
\E(t)=\|u(t)\|^2_1+\|\partial_t u(t)\|^2_1
+\|\partial_{tt}u(t) \|^2 + \int_0^t -g'(s) \|u(t)-u(t-s)\|^2_1ds.
\end{equation}
\end{definition}

The first theorem concerns with the well-posedness of the problem.
As previously mentioned in the Introduction, such a result is not obvious at all, due to the lack of convexity
of the memory kernel $g$.

\begin{theorem}
\label{EU}
For every set of initial data $(u_0,v_0,w_0)\in\H$, problem~\eqref{volterra}-\eqref{IC}
admits a unique global solution (i.e., defined for all $T>0$) in the sense of Definition~\ref{solution}.
Moreover,
there exists $M\geq 1$, independent of the initial data, such that
\begin{equation}
\label{BDD}
\E(t)\leq M\E(0).
\end{equation}
\end{theorem}

The next result is the exponential decay of the energy. We remark that such a decay is obtained under the sole assumption
that the memory kernel is exponentially decaying as well.

\begin{theorem}
\label{EXP}
Assume that there exist $M_g>0$ and $\omega_g>0$ such that
\begin{equation}
\label{gexp}
g(s)\leq M_g\e^{-\omega_g s}.
\end{equation}
Then there exist $M\geq 1$ and $\omega>0$, both independent of the initial data,
such that the energy fulfills the exponential decay estimate
\begin{equation}
\label{DECAY}
\E(t)\leq M\E(0)\e^{-\omega t}.
\end{equation}
\end{theorem}

\smallskip
The kernel $g$ of Example~\ref{exa}, although not convex, fulfills the
condition~\eqref{Dafermos} employed in the earlier literature. Hence, in particular,
it fulfills~\eqref{gexp}, so that Theorem~\ref{EXP} applies.
Nonetheless, it is possible to exhibit a nonconvex $g$ satisfying (g1)-(g3) and \eqref{gexp},
but not~\eqref{Dafermos}, as the following construction shows.

\begin{example}
Without loss of generality, condition (g3) is here assumed with $\alpha-\delta=1$.
Besides, we take $\gamma\geq 1$.
Let $f$ be an absolutely continuous positive function satisfying
the following properties:
\begin{itemize}
\item[(i)] $0<f(s)\leq \e^{-s}$.
\smallskip
\item[(ii)] $f(s)=\varepsilon_n+\varepsilon_n(s-n^2)$ on any interval $I_n=[n^2,n^2+n+1]$ with $n\geq 2$ integer,
where $\varepsilon_n=\frac{\e^{-(n^2+n+1)}}{n+2}$.
\smallskip
\item[(iii)] $f$ is decreasing outside the intervals $I_n$.
\end{itemize}
Then, define
$$g(s)=\int_s^\infty f(y)dy.$$
Observe that $g$ is not convex, since its derivative $g'=-f$ is decreasing on $I_n$.
It is also clear that (g1)-(g2) hold, and the same can be said
of \eqref{gexp}, due to (i). Condition (g3) is trivially verified outside $I_n$,
where $g'$ is negative and increasing, while on $I_n$ we have
$$g'(s) - g''(s) \leq 0=-\big[\varepsilon_n+\varepsilon_n(s-n^2)\big]+\varepsilon_n
\leq 0.$$
We are left to show that \eqref{Dafermos} fails to hold.
By contradiction, suppose \eqref{Dafermos} true on every interval $[n^2,n^2+1]\subset I_n$.
In which case, for any $\tau\in[0,1]$, we have
$$g'(n^2+\tau)+\delta g(n^2+\tau)
=-\varepsilon_n(1+\tau) +\delta\int_{n^2+\tau}^\infty f(y)dy\leq 0.$$
But
$$\int_{n^2+\tau}^\infty f(y)dy\geq \int_{n^2+1}^{n^2+n+1} f(y)dy
\geq n\varepsilon_n.
$$
We conclude that
$$(\delta n-2)\varepsilon_n\leq -\varepsilon_n(1+\tau)+\delta n \varepsilon_n\leq 0,
$$
which is clearly false as $n\to\infty$.
\end{example}

\smallskip
The rest of the paper is devoted to the proofs of Theorem~\ref{EU} and Theorem~\ref{EXP}.
In some sense, we may say that the strategy behind the proofs is classical: we use suitable multipliers
in order to obtain differential inequalities.
However, the presence of the convolution integral introduces essential difficulties.
Indeed, when performing the estimates, integration by parts has to be exploited
several times. This produces boundary terms that cannot be possibly handled in any way.
The only possibility is then
working within a regularization scheme, where these boundary terms simply disappear.
Accordingly, we will consider a family of regularized problems, depending on a small parameter $\varrho>0$,
which converge to \eqref{volterra}-\eqref{IC} in the limit $\varrho\to 0$.
In particular, this will allow us to take advantage of some general technical tools developed in the
study of equations with memory in the past history framework.
%%%%%%%%%%%%%%%%%%%%%%%%%%%%%%%%%%%%%%%%%%%%

%%%%%%%%%%%%%%%%%%%%%%%%%%%%%%%%%%%%%%%%%%%%
\section{The Right-Translation Semigroup on the Memory Space}
\label{SRT}

\noindent
Within assumptions (g1)-(g3), we begin to analyze the strongly continuous semigroup of right translations
on a suitable Lebesgue space. This will play a crucial role in the forthcoming proofs.
To this end, let us introduce the space
$$ \M= L^2_{-g'}(\R^+;H^1),$$
namely, the space of square summable functions $\eta:\R^+\to H^1$ with respect to the
measure $-g'(s)ds$, endowed with the Hilbert scalar product and norm
$$\l\eta_1,\eta_2\r_\M=
\int_0^{\infty} -g'(s) \l\eta_1(s),\eta_2(s)\r_1 ds,
\qquad\| \eta \|_{\M}^2 =\int_0^{\infty} -g'(s) \| \eta(s) \|^2_1ds.$$
Next, we define the linear operator
$\T:\M\to\M$ acting as
$$\T\eta=-\eta',$$
with domain
$$\D(\T)=\big\{\eta\in\M:\,\eta'\in\M,\,\,\eta(0)=0\big\},$$
where $\eta'$ denotes the weak derivative of $\eta=\eta(s)$ with respect to $s$,
and $\eta(0)$ stands for the limit in $H^1$ of $\eta(s)$ as $s\to 0$.

Before going any further, let us discover an important feature of $\T$.

\begin{lemma}
\label{quasidissi}
For every $\eta\in\D(\T)$ we have the equality
$$
\l \T\eta, \eta\r_{\M} = -\frac12\int_0^{\infty} g''(s) \| \eta(s) \|^2_1 ds.
$$
In particular, this means that the integral in the right-hand side exists whenever $\eta\in\D(\T)$.
\end{lemma}

\begin{proof}
Since $\eta\in\D(\T)$, the product $\l \T\eta, \eta\r_{\M}$ makes sense.
By definition,
$$
\l \T\eta, \eta \r_{\M}
= \int_0^{\infty} g'(s) \frac{1}{2}\frac{d}{ds}\|\eta(s)\|^2_1 ds
=\lim_{\nu\to 0,\,N\to\infty} \int_\nu^{N} g'(s) \frac{1}{2}\frac{d}{ds}\|\eta(s)\|^2_1 ds.
$$
Therefore, integrating by parts,
$$
\l \T\eta, \eta \r_{\M}=\lim_{\nu\to 0,\,N\to\infty} \frac{1}{2} \bigg[ g'(N) \| \eta(N) \|_1^2-g'(\nu)\|\eta(\nu)\|_1^2
-\int_\nu^{N} g''(s) \| \eta(s) \|_1^2 ds \bigg].
$$
The conclusion is a consequence of the two limits
$$
\lim_{\nu\to 0}g'(\nu)\| \eta(\nu) \|^2_1 =0\qquad\text{and}\qquad
\lim_{N\to\infty} g'(N)\|\eta(N)\|^2_1=0.
$$
To prove the first limit, we use the H\"older inequality and \eqref{gform}, to get
\begin{align*}
-g'(\nu)\| \eta(\nu) \|_1^2
&\leq -g'(\nu)\bigg [ \int_0^\nu \|\eta'(y)\|_1 dy\bigg]^2\\
&\leq \nu \int_0^\nu -g'(y+\nu-y)\|\eta'(y)\|_1^2 dy\\
&\leq \nu \int_0^\nu -g'(y) \e^{(\alpha-\delta)(\nu-y)}\|\eta'(y)\|_1^2 dy\\
\noalign{\vskip1mm}
&\leq \nu \|\T\eta\|_\M^2\e^{(\alpha-\delta)\nu},
\end{align*}
which tends to zero as $\nu\to 0$.
Hence,
$$\l \T\eta, \eta \r_{\M}
=\lim_{N\to\infty} \frac{1}{2} \bigg[ g'(N) \| \eta(N) \|_1^2
-\int_0^{N} g''(s) \| \eta(s) \|_1^2 ds \bigg].$$
Since the integral
$$\int_0^{\infty} g'(s) \| \eta(s) \|_1^2 ds$$
exists,
setting $h(s)=(\alpha-\delta)g'(s)-g''(s)$,
we learn that the limit
$$\lim_{N\to\infty} \bigg[g'(N) \| \eta(N) \|_1^2
+\int_0^{N} h(s) \| \eta(s) \|_1^2 ds \bigg]$$
exists as well.
On the other hand, we know from (g3) that $h(s)\leq 0$. Accordingly,
both terms in the limit above are negative, hence the integral converges, being
monotone decreasing with respect to $N\to\infty$. In turn
$$\lim_{N\to\infty} g'(N) \| \eta(N) \|_1^2$$
must exists.
But the function $g'(s) \| \eta(s) \|_1^2$ is summable on $\R^+$,
which forces the latter limit to be zero, as desired.
\end{proof}

Our aim is showing that $\T$ generates
the \emph{right-translation semigroup} $R(t)$ on $\M$, defined as
$$
[R(t)\eta](s)
=\begin{cases}
0 &s \leq t,\\
\eta(s-t)  &s > t.
\end{cases}
$$
Checking that $R(t)$ is a strongly continuous semigroup of linear operators is standard matter:
the only nontrivial thing to observe is that $R(t)$ maps $\M$ into $\M$. Indeed,
in light of \eqref{gform},
$$\|R(t)\eta\|_\M^2=\int_t^\infty -g'(s)\|\eta(s-t)\|_1^2 ds
=\int_0^\infty -g'(s+t)\|\eta(s)\|_1^2 ds\leq \|\eta\|_\M^2 \e^{(\alpha-\delta)t}.
$$
The characterization of the infinitesimal generator of $R(t)$
is very well known when
the kernel appearing in the definition of $\M$ is a \emph{decreasing} function. In which case, $R(t)$ turns
out to be a contraction semigroup (see, e.g., \cite{Terreni}). However, here the kernel is
$-g'(s)$, which in general is \emph{not} decreasing.
This translates into the fact that $R(t)$ is not a contraction semigroup (hence its infinitesimal generator
is not dissipative).
Nonetheless, the following holds.

\begin{theorem}
\label{thminf}
The infinitesimal generator of the right-translation semigroup $R(t)$
on $\M$ is the linear operator $\T$ defined above.
\end{theorem}

\begin{proof}
Setting
$$\omega=\frac{\alpha-\delta}2>0,$$
we will equivalently prove
that the linear operator
$$\S=\T-\omega\I,\qquad \D(\S)=\D(\T),$$
is the infinitesimal generator of the semigroup
$$R_*(t)=\e^{-\omega t}R(t).$$
We begin to show that $\S$ generates a contraction semigroup $S(t)$.
This will be done via the Lumer-Phillips theorem (see \cite{PAZ}),
by completing the following two steps.
\begin{itemize}
\item[(i)] $\S$ is dissipative, i.e., $\l\S\eta,\eta\r_\M\leq 0$ for all $\eta\in\D(\S)$.
\smallskip
\item[(ii)] ${\rm Range}(\I-\S)=\M$.
\end{itemize}
Point (i) is an immediate consequence of Lemma~\ref{quasidissi} and (g3).
Indeed, given any $\eta\in\D(\S)$,
$$\l\S\eta,\eta\r_\M=
\frac12\int_0^\infty \big[(\alpha - \delta)g'(s) - g''(s)\big]\|\eta(s)\|_1^2ds\leq 0.
$$
To prove (ii), let $\xi\in\M$ be arbitrarily fixed. We verify that the equation
$$\eta-\S\eta=(1+\omega)\eta+\eta'=\xi$$
has a solution $\eta\in\D(\S)$.
A straightforward integration, along with the requirement $\eta(0)=0$, entail
$$\eta(s) = \int_0^s \e^{-(1+\omega)(s-y)} \xi(y)dy.$$
We need to prove that such an $\eta$ belongs to $\M$. Then,
$\eta'\in\M$ is obtained by comparison, and we conclude that $\eta\in\D(\S)$.
To this aim, an exploitation of \eqref{gform} yields
\begin{align*}
\sqrt{-g'(s)}\|\eta(s)\|_1
&\leq  \int_0^s \e^{-(s-y)} \sqrt{-g'(s)\e^{-2\omega(s-y)}}
\|\xi(y)\|_1dy\\
&\leq  \int_0^s \e^{-(s-y)} \sqrt{-g'(y)}\|\xi(y)\|_1dy.
\end{align*}
Denoting
$${\mathcal E}(s) = \e^{-s}\qquad\text{and}\qquad \phi(s) = \sqrt{-g'(s)}\|\xi(s)\|_1,$$
the latter integral is nothing but the convolution product\footnote{The convolution product $f_1* f_2$
in $\R^+$ can be viewed as the convolution product $\hat f_1*\hat f_2$ in $\R$, where $\hat f_j$
is the null extension of $f_j$ on $\R^-$.} in $\R^+$
$$({\mathcal E} * \phi)(s)=\int_0^s {\mathcal E}(s-y)\phi(y)dy.$$
Since
$$\|{\mathcal E}\|_{L^1(\R^+)} = 1\qquad\text{and}\qquad
\|\phi\|_{L^2(\R^+)}=\|\xi\|_{\M},$$
by a classical result of convolution theory (see, e.g.\ \cite{HS}), we draw the estimate
$$\|\eta\|_\M\leq \|{\mathcal E} * \phi\|_{L^2(\R^+)}\leq
\|{\mathcal E}\|_{L^1(\R^+)}\|\phi\|_{L^2(\R^+)}=\|\xi\|_\M,
$$
allowing us to conclude that $\eta\in\M$.
So far we only proved that $\S$ is the infinitesimal generator
of a contraction semigroup $S(t)$.
In order to prove the equality $S(t)=R_*(t)$,
it is enough showing that
the infinitesimal generator of the semigroup $R_*(t)$, call it $\S_*$,
is an extension of $\S$.
Fix then $\eta\in\D(\S)$. We aim to prove that
$$\exists\,\,\M\text{-}\lim_{t \to 0}\frac{R_*(t)\eta - \eta}{t} \quad(\text{hence equal to }\S_*\eta).$$
Explicitly,
$$\frac{[R_*(t)\eta](s) - \eta(s)}{t}=
\begin{cases}
-\frac{\eta(s)}{t} &s \leq t, \\
\noalign{\vskip1mm}
\frac{\e^{-\omega t}\eta(s-t) - \eta(s)}{t} &s > t.
\end{cases}
$$
By the very definition of $\D(\S)$,
$$\M\text{-}\lim_{t \to 0}
\frac{\e^{-\omega t}\eta(s-t) - \eta(s)}{t}\chi_{(t,\infty)} = -\eta'(s)-\omega\eta(s).
$$
Hence, we are done if we establish the convergence as $t\to 0$
$$-\frac{\eta(s)}{t}\chi_{(0,t)} \to 0 \quad\text{in }\M.
$$
Indeed, making use of the H\"older inequality together with \eqref{gform},
\begin{align*}
\frac{1}{t^2} \int_0^t -g'(s) \|\eta(s)\|^2_1
&\leq \frac{1}{t} \int_0^t -g'(s)\bigg [ \int_0^s \|\eta'(y)\|_1^2 dy\bigg] ds \\
&=\frac{1}{t} \int_0^t \|\eta'(y)\|^2_1\bigg[ \int_y^t -g'(s) ds\bigg] dy \\
&\leq\frac{1}{t} \int_0^t \|\eta'(y)\|^2_1\bigg[ \int_y^t -g'(y)\e^{2\omega(s-y)} ds\bigg] dy \\
&\leq\e^{2\omega t}\int_0^t -g'(y)\|\eta'(y)\|^2_1 dy.
\end{align*}
Since $\eta'\in\M$, the latter term goes to zero when $t\to 0$, as desired.
It is worth noting that, \emph{a posteriori}, we learn that
$R_*(t)=\e^{-\omega t}R(t)$ is a contraction semigroup.
\end{proof}

\subsection*{The nonhomogeneous Cauchy problem}
For a fixed $T>0$, let $f\in L^1(0,T;H^1)$. Consider the nonhomogeneous
Cauchy problem in $\M$
for the unknown variable $\eta^t=\eta^t(s)$ with initial datum $\eta_0=\eta_0(s)\in\M$
$$
\begin{cases}
\displaystyle\frac{d}{dt} \eta^t = \T \eta^t + f(t), &0<t\leq T, \\
\noalign{\vskip1mm}
\eta^0 = \eta_0.
\end{cases}
$$
Note that $f$ can be seen as an element (independent of $s$) of $L^1(0,T;\M)$.
It is then well known that the problem above admits a unique mild solution on $[0,T]$ given by
$$\eta^t=R(t)\eta_0+\int_0^t R(t-\tau)f(\tau) d\tau.$$
Explicitly, we have the representation formula
$$
\eta^t(s) =
\begin{cases}
\displaystyle\int_{t-s}^t f(y)dy &s \leq t,\\
\displaystyle\eta_{0}(s-t) + \int_{0}^t f(y)dy &s > t.
\end{cases}
$$
Moreover, if $f\in \C([0,T],H^1)$ and $\eta_0\in\D(\T)$, then, due to the particular form of the semigroup,
we can conclude that $\eta^t\in \D(\T)$ for every $t\in[0,T]$. Hence, $\eta^t$ is actually
a strong solution (see \cite[Ch.4, Theorem 2.4]{PAZ}).
%%%%%%%%%%%%%%%%%%%%%%%%%%%%%%%%%%%%%%%%%%%%

%%%%%%%%%%%%%%%%%%%%%%%%%%%%%%%%%%%%%%%%%%%%
\section{The Approximated Problem}
\label{SAP}

\noindent
Let (g1)-(g3) and \eqref{subcritical} hold.
For an arbitrarily fixed $\varrho>0$, we set
$$
\q_{\varrho}(s) =
\begin{cases}
0 &\mathrm{if} \ s < \varrho /2, \\
\frac{2s}{\varrho} -1   &\mathrm{if} \ \varrho/2 \leq s \leq \varrho, \\
1 &\mathrm{if} \ s > \varrho,
\end{cases}
$$
and we define the (positive) function
$$\Q_{\varrho}(t)=\int_{t}^{\infty} g(s)[1-\q_{\varrho}(s-t)] ds. $$
Observe that, as $g$ is decreasing, we have the bound
\begin{equation}
\label{Qbound}
\Q_{\varrho}(t) \leq \varrho g(t).
\end{equation}
Then, for any initial data $(u_0, v_0, w_0)\in \H$, we analyze the Cauchy problem
\begin{equation}
\label{volterra_rho}
\begin{cases}
\displaystyle
\partial_{ttt} u(t) + \alpha \partial_{tt} u(t) + \beta A\partial_{t} u(t) + \gamma Au(t) - \int_{0}^{t}g(s) Au(t-s)ds
=\Q_{\varrho}(t)Au_0,\\
\noalign{\vskip-1mm}
u(0) = u_0,\\
\noalign{\vskip.7mm}
\partial_t u(0) = v_0,\\
\noalign{\vskip.7mm}
\partial_{tt} u(0)=w_0.
\end{cases}
\end{equation}
This is a sort of a nonhomogeneous version of our original problem \eqref{volterra}-\eqref{IC}, with a
forcing term given by $\Q_{\varrho}Au_0$, depending itself on
(the first component of) the initial data.
On account of \eqref{Qbound}, problem \eqref{volterra}-\eqref{IC}
is formally recovered in the limit $\varrho\to 0$.

\smallskip
Defining, for $t\geq 0$ and $s>0$, the function
\begin{equation}
\label{eta_repr}
\eta^t(s) = \begin{cases} u(t)-u(t-s) &s \leq t,\\
u(t)+[\q_{\varrho}(s-t) -1]u_0 &s > t,
\end{cases}
\end{equation}
the first equation of \eqref{volterra_rho} becomes
\begin{equation}
\label{voltapprox}
\partial_{ttt} u(t) + \alpha \partial_{tt} u(t) + \beta A\partial_{t} u(t) + (\gamma-\kappa) Au(t)
+ \int_{0}^{\infty}g(s)A\eta^t(s) ds = 0.
\end{equation}
Indeed, using (g1),
$$\int_{0}^{\infty}g(s)A\eta^t(s) ds=\kappa Au(t)
-\int_{0}^tg(s)Au(t-s) ds+Au_0
\int_{t}^{\infty}g(s)[\q_{\varrho}(s-t) -1] ds.
$$From our discussion in Section~\ref{SRT},
assuming $u$ to be known and such that $\partial_t u\in \C([0,T],H^1)$,
the function $\eta^t$ defined via the representation formula~\eqref{eta_repr}
is exactly the strong solution on $[0,T]$ of the nonhomogeneous differential equation in $\M$
\begin{equation}
\label{eta}
\frac{d}{dt} \eta^t = \T \eta^t + \partial_t u(t),
\end{equation}
with initial datum
$$\eta^0(s)=\q_{\varrho}(s) u_0,$$
which is readily seen to belong to $\D(\T)$.

\begin{remark}
In particular, $\eta^t\in\D(\T)$ for all $t$ and Lemma~\ref{quasidissi} applies, so we have
\begin{equation}
\label{WOW}
\l \T\eta^t, \eta^t \r_{\M} = -\frac12\int_0^{\infty} g''(s) \| \eta^t(s) \|^2_1 ds.
\end{equation}
This is exactly the reason why we introduced the approximating $\varrho$-problem:
if $\varrho=0$, and the original equation is recovered,
the initial datum of $\eta^t$ becomes the constant function $u_0$, which \emph{does not} belong to $\D(\T)$, unless
$u_0=0$.
\end{remark}

In spite of the complete equivalence between \eqref{volterra_rho}
and \eqref{voltapprox},
the latter formulation is much more efficient in order
to obtain suitable bounds on the solutions (see the next Section~\ref{SAPEI}).
To this purpose,
we define the energy of the approximated problem to be
\begin{equation}
\label{energyrho}
\E_\varrho(t)=\|u(t)\|^2_1+\|\partial_t u(t)\|^2_1
+\|\partial_{tt}u(t) \|^2 + \|\eta^t\|^2_\M.
\end{equation}
Here, we highlighted the dependence on $\varrho$, in order not to confuse
this energy with the energy $\E(t)$ of the original equation~\eqref{volterra}.
In particular, since $\eta^0(s)=\q_{\varrho}(s)u_0$, and recalling Remark~\ref{Remg1},
\begin{equation}
\label{energyrho1}
\E_\varrho(0)\leq [1+g(0)]\|u_0\|^2_1+\|v_0\|^2_1+\|w_0\|^2.
\end{equation}

%%%%%%%%%%%%%%%%%%%%%%%%%%%%%%%%%%%%%%%%%%%%

%%%%%%%%%%%%%%%%%%%%%%%%%%%%%%%%%%%%%%%%%%%%
\section{The Approximated Problem: Existence and Uniqueness}
\label{SAPEU}

\noindent
The following theorem holds.

\begin{theorem}
\label{exist_unique}
Let $\varrho>0$ be fixed.
For every $T>0$ and every initial data $(u_0, v_0, w_0)\in \H$, the Cauchy problem
\eqref{volterra_rho} admits a unique weak solution $u$ on the time interval $[0,T]$,
possessing the regularity
$$(u,\partial_t u,\partial_{tt}u) \in \C([0,T],\H).$$
\end{theorem}

\medskip
\noindent
\emph{Proof.}
The argument, based on a Galerkin approximation procedure, is carried out in four steps.
The main ingredient is the
basic energy inequality \eqref{Ebdd} of Remark~\ref{RemEbdd}, proved in the next section.

\medskip
\noindent
{\bf I. Galerkin approximations.}
Let $\{x_j\}_{j=1}^\infty\subset H^2$ be an orthonormal basis of $H$.
For every $n$, we define the finite-dimensional space
$H_n = \mathrm{span}\{x_1, ..., x_n\}$,
and we denote by $P_n:H \rightarrow H_n$ the orthogonal projection of $H$ onto $H_n$.
We approximate the given initial data $(u_{0},v_{0},w_{0})\in\H$ with a
sequence
$$(u_{0 n},v_{0 n},w_{0 n})=(P_n u_{0},P_n v_{0},P_n w_{0})\to
(u_{0},v_{0},w_{0})\quad \text{in } \H.
$$
Then, there exists a function $u_n\in\C^3([0,T],H_n)$
of the form
$$u_n(t) = \sum_{j=1}^n a_j^n(t)x_j$$
satisfying,
for every test function $\zeta\in H^1$ and every $t \in [0,T]$,
\begin{align}
\label{weak_sol_n}
&\l \partial_{ttt}u_n(t), \zeta \r+ \alpha \l  \partial_{tt}u_n(t), \zeta \r
+ \beta \l  \partial_t u_n(t), \zeta \r_1 + \gamma \l  u_n(t), \zeta \r_1\\
&\quad - \int_{0}^{t}g(s)\l u_n(t-s), \zeta \r_1 ds
= \Q_{\varrho}(t)\l u_{0n},\zeta\r_1,\notag
\end{align}
along with the initial conditions
\begin{equation}
\label{initial_cond}
\begin{cases}
u_n(0) = u_{0 n}, \\
\partial_t u_n(0) = v_{0 n}, \\
\partial_{tt} u_n(0) = w_{0 n}.
\end{cases}
\end{equation}
Indeed, choosing $\zeta=x_j$ for any $j\in\{1,\ldots,n\}$, everything boils down to
a system of $n$ linear ordinary integro-differential equations\footnote{The equations are uncoupled if
the vectors $x_j$ are also orthogonal in $H^1$.
This occurs for instance if $A$ has compact inverse. In which case, the $x_j$ are chosen to be the eigenvectors of $A$.}
in the unknowns $a_j^n$,
and the existence (and uniqueness) of a solution on $[0,T]$ is guaranteed
by a classical fixed point argument.
Moreover,
calling for short
$z_n=(u_n, \partial_t u_n, \partial_{tt} u_n)$, and making use
of \eqref{Ebdd}, with the aid of~\eqref{energyrho1}, we have the uniform estimate
$$
\max_{t\in [0,T]}\|z_n(t)\|_{\H}\leq C\|z(0)\|_{\H},
$$
for some $C\geq 1$ independent of $n$, $\varrho$ and of the initial data.
Here, it is important to point out that \eqref{Ebdd} is obtained by considering the equation
for $u_n$ in the form \eqref{voltapprox}.
\qed

\medskip
\noindent
{\bf II. Passage to the limit.}
We claim that $z_n$ is a Cauchy sequence in $\C([0,T],\H)$.
The difference $z_n-z_m$
solves exactly \eqref{weak_sol_n}-\eqref{initial_cond},
but where every $n$-object is replaced by the difference of the corresponding
$n$-object and the $m$-object. Hence, a further
application of \eqref{Ebdd} yields
$$
\max_{t\in [0,T]}\|z_n(t)- z_m(t)\|_{\H} \leq C \|z_n(0)- z_m(0)\|_{\H},
$$
and the latter norm goes to zero by construction when $n,m\to\infty$.
Accordingly, $z_n \to z$ in $\C([0,T],\H)$ for some $z$, which
is readily seen to be of the form
$z=(u,\partial_t u,\partial_{tt} u)$.
\qed

\medskip
\noindent
{\bf III. Existence.}
Let us prove that $u$ is a solution to the Cauchy problem \eqref{volterra_rho}.
First, we observe that the desired initial conditions are satisfied by continuity.
Thus, we are left to verify
that the differential equation in \eqref{volterra_rho} holds true (in the weak sense).
This is obtained by letting $n\to\infty$ in \eqref{weak_sol_n}, just noting that
$$\Q_{\varrho}(t)\l u_{0n}, \zeta \r_1\to \Q_{\varrho}(t)\l u_0, \zeta \r_1,$$
for any given $\zeta \in H^1$ and every $t\in[0,T]$.
\qed

\medskip
\noindent
{\bf IV. Uniqueness.}
Let $u^1$ and $u^2$ be any two solutions to~\eqref{volterra_rho} originating
from the same initial data. Their difference
$u^1-u^2$
solves the Cauchy problem \eqref{volterra_rho} with null initial data.
Invoking once more \eqref{Ebdd}, we obtain that $u^1(t)=u^2(t)$ for all $t\in [0,T]$.
\qed
%%%%%%%%%%%%%%%%%%%%%%%%%%%%%%%%%%%%%%%%%%%%

%%%%%%%%%%%%%%%%%%%%%%%%%%%%%%%%%%%%%%%%%%%%
\section{The Approximated Problem: The Energy Inequality}
\label{SAPEI}

\noindent
Let $\varrho>0$ be fixed. Along this section,
$u(t)$ will be any solution to~\eqref{volterra_rho},
while $C$ will always denote a generic positive constant independent
of $\varrho$ and of the particular $u(t)$.
Our goal is to establish the dissipative character of~\eqref{volterra_rho}.
To this end, considering the equation in the form~\eqref{voltapprox} with $\eta^t$
given by~\eqref{eta_repr},
we introduce the energy-like functional
\begin{align*}
\F_{\varrho}(t)
&=\frac{\gamma-\kappa}{\alpha}\|\partial_t u(t) + \alpha
u(t)\|_1^2
+ \frac{\varkappa\alpha+\kappa}{\alpha} \| \partial_t u(t) \|_1^2
+ \|\partial_{tt} u(t) + \alpha \partial_t u(t)\|^2\\
&\quad +\| \eta^t \|_{\M}^2 +\alpha \int_{0}^{\infty} g(s) \| \eta^t(s) \|_1^2 ds
+ 2\int_{0}^{\infty} g(s) \l \eta^t(s), \partial_t u(t) \r_1 ds,
\end{align*}
keeping in mind that $\gamma>\kappa$ and $\varkappa>0$.
Offhand, it is not even clear whether $\F_{\varrho}(t)$ is well defined on the trajectories
of~\eqref{volterra_rho}. Indeed, this will be a consequence of the forthcoming
estimates. First, let us show that $\F_{\varrho}(t)$ is ``almost" equivalent to the energy $\E_{\varrho}(t)$
given by~\eqref{energyrho}.

\begin{lemma}
\label{boundEF}
There exists $C\geq 1$ such that
\begin{equation}
\label{INNE}
\frac1C\E_{\varrho}(t) \leq \F_{\varrho}(t)\leq C\bigg(\E_{\varrho}(t)+\int_{0}^{\infty} g(s) \|\eta^t(s)\|_1^2 ds\bigg).
\end{equation}
Moreover, for a possibly bigger $C$,
\begin{equation}
\label{INNE2}
\F_{\varrho}(0)\leq C \E_{\varrho}(0).
\end{equation}
\end{lemma}

\begin{proof}
By means of the H\"older and the Young inequalities,
$$
\bigg |2\int_{0}^{\infty} g(s) \l \eta^t(s), \partial_t u(t) \r_1 ds \bigg|
\leq \alpha \int_{0}^{\infty}g(s)\|\eta^t(s)\|^2_1 +\frac{\kappa}{\alpha}\|\partial_t u(t)\|^2_1.
$$
Accordingly
\begin{equation*}
\F_{\varrho}(t) \geq \frac{\gamma - \kappa}{\alpha} \|\partial_t u(t) + \alpha
u(t)\|_1^2 +\varkappa \|\partial_t u(t)\|_1^2+ \|\partial_{tt} u(t) + \alpha \partial_t u(t)\|^2 + \|\eta^t\|^2_{\M}.
\end{equation*}
Since
\begin{equation}
\label{equinorma}
\|(u,v,w)\|_{*}^2 = \frac{\gamma - \kappa}{\alpha}\|v+\alpha u\|_1^2 + \varkappa \|v\|_1^2+\|w+\alpha v\|^2
\end{equation}
is readily seen to be an equivalent norm on $\H$, we obtain the first bound in~\eqref{INNE}.
The second bound follows directly by the H\"older and the Young inequalities,
together with the continuous embedding $H^1\subset H$.
Finally, recalling that $\eta^0(s)=\q_{\varrho}(s)u_0$,
and using the fact that $0\leq \q_{\varrho}\leq 1$, we get
$$\int_{0}^{\infty} g(s) \|\eta^0(s)\|_1^2 ds \leq \kappa \|u_0\|^2_1\leq \kappa \E_\varrho(0),$$
implying the desired control~\eqref{INNE2} at $t=0$.
\end{proof}

\begin{remark}
If the classical condition \eqref{Dafermos} on the kernel is also assumed, it is immediately seen that
$\F_{\varrho}(t)$ and $\E_{\varrho}(t)$ are actually equivalent.
\end{remark}

The next step is showing that $\F_{\varrho}(t)$ is a decreasing function.
This will be attained via a differential inequality. Such an inequality
will make sense for the Galerkin approximations of $u(t)$, i.e., the solutions $u_n(t)$ to \eqref{weak_sol_n}-\eqref{initial_cond}.
Indeed, the scope of the inequality is twofold: first,
to produce a uniform estimate on the Galerkin approximations, and only in a second moment to extend the very same estimate
to the actual solution $u(t)$. Accordingly, it is understood that all the multiplications hereafter are carried out
within the approximation scheme.

\begin{theorem}
\label{BASIC}
The following inequality holds:
\begin{equation}
\label{diss_estimate}
\frac{d}{dt} \F_{\varrho}(t)+2\varkappa \alpha \| \partial_t u(t) \|_1^2 +\delta\|\eta^t \|^2_\M\leq 0.
\end{equation}
In particular, $\F_{\varrho}(t)$ is decreasing.
\end{theorem}

\begin{proof}
In light of the equivalence between the differential equation in \eqref{volterra_rho} and
\eqref{voltapprox}, we
multiply \eqref{voltapprox} by $\partial_{tt} u + \alpha \partial_{t} u$ in $H$, to get
(omitting everywhere the dependence on $t$)
\begin{align}
\label{main_eq}
&\frac{d}{dt} \Big(\frac{\gamma-\kappa}{\alpha}\|\partial_t u+\alpha u\|_1^2
+ \frac{\varkappa\alpha+\kappa}{\alpha} \| \partial_t u \|_1^2
+ \|\partial_{tt} u+\alpha \partial_t u\|^2\Big)\\
\noalign{\vskip1mm}
&\quad+2(\varkappa \alpha + \kappa)  \| \partial_t u \|_1^2={\mathfrak I}_1 + {\mathfrak I}_2,\notag
\end{align}
where
$${\mathfrak I}_1 =-2\int_{0}^{\infty} g(s)\l \eta(s), \alpha \partial_{t} u\r_1  ds
\qquad\text{and}\qquad
{\mathfrak I}_2=-2\int_{0}^{\infty} g(s) \l \eta(s), \partial_{tt} u\r_1  ds.$$
Let us examine ${\mathfrak I}_1$. Exploiting~\eqref{eta},  we have
$${\mathfrak I}_1 =-\frac{d}{dt} \bigg (\alpha \int_0^{\infty} g(s)\|\eta(s)\|^2_1 ds\bigg)
+ 2\alpha \int_0^{\infty} g(s) \l \T \eta(s), \eta(s) \r_1 ds.
$$
Besides,
$$2\alpha\int_0^{\infty} g(s) \l \T \eta(s), \eta(s) \r_1 ds
=- \alpha\int_0^{\infty} g(s)\frac{d}{ds}\|\eta(s)\|^2_1 ds=\alpha\int_0^{\infty} g'(s) \|\eta(s)\|^2_1 ds.
$$
The boundary terms of the latter integration by parts vanish.
Indeed, for any fixed $t>0$, we infer from the representation formula \eqref{eta_repr}
that
$$\|\eta^t(N)\|_1=\|u(t)\|_1\qquad\text{and}\qquad
\|\eta^t(\nu)\|_1 = \|u(t)-u(t-\nu)\|_1,$$
whenever $N$ is large enough and $\nu$ small enough. Knowing that
$g(\infty)=0$ and $g(0)<\infty$, and by the time-continuity of $u$,
we conclude that
$$\lim_{N\to\infty} g(N)\|\eta(N)\|_1^2=0\qquad\text{and}\qquad
\lim_{\nu\to 0} g(\nu)\|\eta(\nu)\|_1^2=0.$$
In summary,
$$
{\mathfrak I}_1 =-\frac{d}{dt} \bigg (\alpha \int_0^{\infty} g(s) \| \eta(s) \|^2_1 ds \bigg)
+\alpha\int_0^{\infty} g'(s) \|\eta(s)\|^2_1 ds.
$$
By the same argument, we find
$$
{\mathfrak I}_2 =-\frac{d}{dt} \bigg(2\int_{0}^{\infty} g(s)\l\eta(s),\partial_{t}u\r_1 ds\bigg)
+ 2\int_0^{\infty} g'(s)\l\eta(s), \partial_{t} u\r_1  ds +2\kappa \| \partial_t u \|_1^2.
$$
Substituting now ${\mathfrak I}_1$ and ${\mathfrak I}_2$ into~\eqref{main_eq},
we arrive at
$$
\frac{d}{dt} \F_{\varrho}+2\varkappa \alpha \| \partial_t u \|_1^2
=\frac{d}{dt}\|\eta\|_{\M}^2+\alpha\int_0^{\infty} g'(s) \|\eta(s)\|^2_1 ds +2\int_0^{\infty} g'(s)\l\eta(s), \partial_{t} u\r_1 ds.
$$
At this point, we multiply equation~\eqref{eta} by $\eta$ in $\M$.
On account of~\eqref{WOW}, this yields
\begin{equation*}
\frac{d}{dt}\|\eta \|^2_{\M}=
-\int_0^{\infty} g''(s)\|\eta(s)\|^2_1 ds-2\int_0^{\infty} g'(s) \l\eta(s),\partial_t u \r_1 ds.
\end{equation*}
Collecting the last two identities, we end up with
$$
\frac{d}{dt} \F_{\varrho}+2\varkappa \alpha \| \partial_t u \|_1^2
= \int_{0}^{\infty}\big[\alpha g'(s)-g''(s)\big] \|\eta(s)\|^2_1 ds.
$$
The desired inequality follows
by assumption (g3) on the memory kernel.
\end{proof}

\begin{remark}
\label{RemEbdd}
In particular, an integration in time of \eqref{diss_estimate} together with
Lemma~\ref{boundEF} entail the (uniform) estimate
\begin{equation}
\label{Ebdd}
\|u(t)\|^2_1+\|\partial_t u(t)\|^2_1
+\|\partial_{tt}u(t) \|^2\leq \E_{\varrho}(t) \leq C\E_{\varrho}(0),
\end{equation}
for every $t\geq 0$ and some $C\geq 1$.
\end{remark}

\begin{remark}
As mentioned in the Introduction, if we are only interested in proving the existence
of a solution on any interval $[0,T]$ to~\eqref{volterra_rho}, and in turn to~\eqref{volterra},
we do not really need assumption~\eqref{subcritical}.
Indeed, if~\eqref{subcritical} fails to hold, by adding the term $m\partial_{tt}u$ to both sides of~\eqref{volterra_rho},
up to taking $m>0$ large enough that
$$\varkappa_m=\beta-\frac{\gamma}{\alpha+m}>0,$$
instead of~\eqref{diss_estimate} we find the estimate
$$
\frac{d}{dt} \F_{\varrho}(t)\leq C\E_\varrho(t),
$$
for some $C>0$. This is enough to conclude that, for a possibly bigger $C$,
$$\|u(t)\|^2_1+\|\partial_t u(t)\|^2_1
+\|\partial_{tt}u(t) \|^2\leq \E_{\varrho}(t) \leq C\E_{\varrho}(0)\e^{CT},$$
for every $t\in[0,T]$.
\end{remark}
%%%%%%%%%%%%%%%%%%%%%%%%%%%%%%%%%%%%%%%%%%%%

%%%%%%%%%%%%%%%%%%%%%%%%%%%%%%%%%%%%%%%%%%%%
\section{The Approximated Problem: Exponential Decay of the Energy}
\label{SAPEDE}

\noindent
We now dwell on the exponential decay of the energy $\E_\varrho(t)$ of the approximated problem. Again, in what follows,
$u(t)$ will be any solution to~\eqref{volterra_rho}, for a fixed $\varrho>0$.

\begin{theorem}
\label{EXPrho}
Let assumption \eqref{gexp} hold. Then there exist
$M_0\geq 1$ and $\omega>0$, both independent of $\varrho$ and of the initial data,
such that the energy~\eqref{energyrho} fulfills the exponential decay estimate
\begin{equation}
\label{DECAYrho}
\E_\varrho(t)\leq M_0\E_\varrho(0)\e^{-\omega t}.
\end{equation}
\end{theorem}

The  proof of the theorem requires a number of passages: the starting point is the dissipative estimate~\eqref{diss_estimate},
along with the controls provided by Lemma~\ref{boundEF}. However, such an estimate alone
is not enough to drive uniformly the energy to zero. Accordingly,
we need to introduce suitable auxiliary functionals, in order to attain an efficient differential inequality.
All the constants appearing in this section, such as the generic constant $C>0$, are understood to be independent
of $\varrho$ and of the particular solution $u(t)$.

\smallskip
With reference to  \eqref{gexp}, let us denote
$$
\Theta(t) = \int_0^t \e^{-\omega_g(t-s)} \|u(s)\|^2_1 ds.
$$
A preliminary lemma is needed.

\begin{lemma}
\label{Lemmafine}
There exists a functional $\Psi(t)$ satisfying the differential inequality
\begin{equation}
\label{Psi-ineq}
\frac{d}{dt} \Psi(t) + \vartheta \F_\varrho(t)
\leq C \big[\| \partial_t u(t)\|^2_1+\|\eta^t\|_\M^2\big],
\end{equation}
along with the bounds
\begin{equation}
\label{Psi-bound}
-C\E_{\varrho}(t)\leq \Psi(t) \leq C\big[\E_\varrho(t)+\E_{\varrho}(0) \e^{-\omega_g t}+\Theta(t)\big],
\end{equation}
for some $\vartheta>0$ and $C>0$.
\end{lemma}

\begin{proof}
Calling
$$G(s)=\int_s^{\infty} g(y)dy,$$
we define our functional $\Psi$ to be the sum
$$\Psi(t)=\Psi_1(t)+\Psi_2(t),$$
where
\begin{align*}
\Psi_1(t)&= -\l \partial_t u(t)-\alpha u(t), \partial_{tt} u(t) + \alpha \partial_t u(t)\r,\\
\noalign{\vskip1mm}
\Psi_2(t)&=\frac\alpha2 \int_0^{\infty} G(s) \|\eta^t(s) - u(t)\|^2_1 ds.
\end{align*}
For some $C>0$, we have the inequalities
\begin{align}
\label{PsiUNO}
|\Psi_1(t)|&\leq C\E_\varrho(t),\\
\noalign{\vskip1mm}
\label{PsiDUE}
0\leq \Psi_2(t)&\leq C\big[\E_{\varrho}(0) \e^{-\omega_g t}+\Theta(t)\big].
\end{align}
Indeed, \eqref{PsiUNO} is straightforward. Concerning \eqref{PsiDUE},
we first observe that the exponential decay~\eqref{gexp} of $g$ implies that
$G(s)\leq M_G\e^{-\omega_g s}$, with $M_G=M_g/\omega_g$.
Hence, writing explicitly $\eta^t-u(t)$ via~\eqref{eta_repr},
\begin{align*}
\frac2\alpha\Psi_2(t)&= \int_t^{\infty} G(s) [\q_{\varrho}(s-t)-1]^2 \| u_0 \|^2_1 ds+\int_0^t G(t-s) \|u(s)\|^2_1 ds\\
&\leq \frac{M_G}{\omega_g}\E_\varrho(0)\e^{-\omega_g t}+ M_G\Theta(t).
\end{align*}
Collecting~\eqref{PsiUNO}-\eqref{PsiDUE},
the bounds~\eqref{Psi-bound} are established.
We are left to prove~\eqref{Psi-ineq}.
Exploiting~\eqref{voltapprox}, we compute the time-derivative of $\Psi_1$ as
(omitting the dependence on time)
\begin{align*}
&\frac{d}{dt} \Psi_1+\frac{\gamma-\kappa}{\alpha}\| \partial_t u + \alpha u \|^2_1+\|\partial_{tt}u + \alpha \partial_t u\|^2\\
\noalign{\vskip1mm}
&=\frac{2\gamma-3\kappa-\varkappa\alpha}{\alpha}\l \partial_t u + \alpha u,\partial_t u \r_1
+2\alpha\l \partial_{tt} u + \alpha \partial_t u,\partial_{t} u \r + \frac{2\kappa+2\varkappa\alpha}{\alpha}\| \partial_t u \|^2_1\\
&\quad +\int_0^{\infty} g(s)\l \eta(s),\partial_t u\r_1 ds- \alpha\int_0^{\infty} g(s)\l \eta(s),u\r_1 ds.
\end{align*}
By the Young and the H\"older inequalities, for some $C>0$ large enough we obtain
\begin{align*}
&\frac{d}{dt} \Psi_1+\frac{\gamma-\kappa}{2\alpha}\| \partial_t u + \alpha u \|^2_1+\frac12 \|\partial_{tt}u +\alpha\partial_t u\|^2\\
\noalign{\vskip1mm}
&\leq C\|\partial_t u\|_1^2+\frac\alpha4 \int_0^{\infty} g(s) \| \eta(s) \|^2_1ds
- \alpha\int_0^{\infty} g(s)\l \eta(s),u\r_1 ds.
\end{align*}
The time-derivative of $\Psi_2$, in light of \eqref{eta}, reads
$$
\frac{d}{dt}\Psi_2 = -\frac\alpha2\int_0^{\infty} g(s) \| \eta(s) \|^2_1ds + \alpha \int_0^{\infty} g(s)\l \eta(s), u \r_1 ds.
$$
Here, an integration by parts is performed. Similarly to the previous cases,
as we are working with the approximated problem, the boundary terms vanish.
Summing the differential inequalities for $\Psi_1$ and $\Psi_2$, we finally get
$$\frac{d}{dt} \Psi+\frac{\gamma-\kappa}{2\alpha}\| \partial_t u + \alpha u \|^2_1+\frac12 \|\partial_{tt}u + \alpha \partial_t u\|^2
+\frac\alpha4 \int_0^{\infty} g(s) \| \eta(s) \|^2_1ds
\leq C\|\partial_t u\|_1^2.
$$
On account of \eqref{INNE}, and recalling that $\|\cdot\|_{*}$ defined in~\eqref{equinorma}
is an equivalent norm on $\H$, the desired conclusion~\eqref{Psi-ineq} follows by adding the term
$\|\partial_t u\|_1^2+\|\eta\|_\M^2$ to both sides of the inequality.
\end{proof}

We can now proceed to the proof of the theorem.

\begin{proof}[Proof of Theorem \ref{EXPrho}]
For $\varepsilon>0$, we define the functional
$$\Lambda(t)=\F_{\varrho}(t) + \varepsilon \Psi(t).$$
Up to choosing $\varepsilon$ small enough, we deduce from~\eqref{Psi-bound}, together with
the first inequality in~\eqref{INNE}, the bounds
\begin{equation}
\label{BNDS}
\varepsilon \E_{\varrho}(t)\leq \Lambda(t) \leq 2\F_\varrho(t)+\E_{\varrho}(0) \e^{-\omega_g t}+\Theta(t).
\end{equation}
Besides, by possibly further reducing $\varepsilon$,
collecting~\eqref{diss_estimate} and~\eqref{Psi-ineq} we obtain
$$\frac{d}{dt}\Lambda(t)+\varepsilon\vartheta\F_{\varrho}(t)\leq 0.
$$
We remark that $\varepsilon$ has been chosen independently of $\varrho$ and $u(t)$.
Leaning again on~\eqref{INNE},
we can find $\mu>0$ such that, for every $\omega>0$ small,
$$\frac{d}{dt}\Lambda(t)+2\omega\F_{\varrho}(t)\leq -\mu\|u(t)\|^2_1.
$$
At this point, we take advantage of~\eqref{BNDS}, arriving at the inequality
$$\frac{d}{dt}\Lambda(t)+\omega\Lambda(t)\leq \omega\E_{\varrho}(0) \e^{-\omega_g t}
+\omega\Theta(t) -\mu\|u(t)\|^2_1.
$$
Taking $\omega<\omega_g$,
an application of the Gronwall lemma, along with the first inequality in~\eqref{BNDS}, yield
$$
\varepsilon \E_{\varrho}(t)\leq \Lambda(t) \leq \Big[\Lambda(0)
+\frac{\omega}{\omega_g-\omega}\E_{\varrho}(0)
+  {\mathfrak I}(t)\Big] \e^{-\omega t},
$$
where
$${\mathfrak I}(t)= \omega \int_0^t \e^{\omega s} \Theta(s) ds -\mu \int_0^t \e^{\omega s} \|u(s)\|^2_1 ds.$$
Observe that, exchanging the order of integration,
$$
\int_0^t \e^{\omega s} \Theta(s) ds
=\int_0^t \e^{\omega_g s}  \|u(s)\|^2_1 \bigg( \int_s^{t} \e^{-(\omega_g-\omega) y} dy \bigg)ds
\leq \frac{1}{\omega_g - \omega} \int_0^t \e^{\omega s} \|u(s)\|^2_1 ds.
$$
Accordingly, choosing $\omega$ small enough that $\omega\leq \mu(\omega_g-\omega)$,
we have that ${\mathfrak I}(t)\leq 0$. Hence, we end up with the energy inequality
$$
\E_{\varrho}(t)\leq \frac1\varepsilon\Lambda(t) \leq
\frac1\varepsilon\big[\Lambda(0)+\mu\E_{\varrho}(0)\big]\e^{-\omega t}.
$$
Noting that $\Lambda(0)$ is bounded
by a multiple of $\E_{\varrho}(0)$, due to \eqref{BNDS} and \eqref{INNE2},
the conclusion follows.
\end{proof}
%%%%%%%%%%%%%%%%%%%%%%%%%%%%%%%%%%%%%%%%%%%%

%%%%%%%%%%%%%%%%%%%%%%%%%%%%%%%%%%%%%%%%%%%%
\section{Proofs of the Main Results}
\label{SAPMR}

\noindent
We have now all the ingredients to prove Theorems~\ref{EU} and~\ref{EXP}.
To this aim, for an arbitrarily given vector
$(u_0,v_0,w_0)\in \H$ and any $\varrho>0$, we denote by
$u_{\varrho }$ the solution to the approximating problem~\eqref{volterra_rho}.

\smallskip
\noindent
{\bf I.}
We claim that, for every $T>0$, the family $z_\varrho=(u_\varrho,\partial_t u_\varrho,\partial_{tt} u_\varrho)$ is Cauchy in $\C([0,T],\H)$
as $\varrho\to 0$. Indeed, take $\varrho_1\leq\varrho_2$, and consider the corresponding
solutions $u_{\varrho_1}$ and $u_{\varrho_2}$
to~\eqref{volterra_rho}.
Observe that the two forcing terms in the equations are different, namely,
$\Q_{\varrho_1}(t)Au_0$ and
$\Q_{\varrho_2}(t)Au_0$.
Setting
$$\bar{u}(t)=u_{\varrho_2}(t)-u_{\varrho_1}(t),$$
we define
$$
{\bar\eta}^t(s) = \begin{cases} \bar{u}(t)-\bar{u}(t-s) &s \leq t,\\
\bar{u}(t)+[\q_{\varrho_2}(s-t)-\q_{\varrho_1}(s-t)]u_0 &s > t.
\end{cases}
$$
Then $\bar{u}$
solves the analogue of equation~\eqref{voltapprox}, i.e.,
$$
\partial_{ttt} \bar u(t) + \alpha \partial_{tt} \bar u(t) + \beta A\partial_{t} \bar u(t)+ (\gamma-\kappa) A\bar u(t)
+ \int_{0}^{\infty}g(s)A{\bar\eta}^t(s) ds = 0,
$$
but with null initial conditions
$$\bar{u}(0)=\partial_t \bar{u}(0)=\partial_{tt}\bar{u}(0)=0.
$$
At this point, defining the energy of the system above as
$$\E_{\varrho_1,\varrho_2}(t)=\|\bar{u}(t)\|^2_1+\|\partial_t \bar{u}(t)\|^2_1
+\|\partial_{tt}\bar{u}(t) \|^2 + \|\bar{\eta}^t\|^2_\M,$$
the very same computations of Section~\ref{SAPEI} lead us to
an inequality completely analogous to~\eqref{Ebdd},
that is,
$$
\E_{\varrho_1,\varrho_2}(t) \leq C\E_{\varrho_1,\varrho_2}(0)=C\|\bar{\eta}^0\|^2_\M,
$$
for some $C\geq 1$ independent of $\varrho_1,\varrho_2$.
Since
$\bar{\eta}^0$
is supported on the interval $[\varrho_1/2,\varrho_2]$ and $|\q_{\varrho_2}-\q_{\varrho_1}|\leq 1$,
$$\|\bar{\eta}^0\|_{\M}^2=\int_{\varrho_1/2}^{\varrho_2}-g'(s)[\q_{\varrho_2}(s)-\q_{\varrho_1}(s)]^2\| u_0\|_1^2 ds\leq
\|u_0\|_1^2 \big[g(\varrho_1/2)-g(\varrho_2)\big]\to 0$$
as $\varrho_1,\varrho_2\to 0$.
This establishes the claim.
\qed

\smallskip
\noindent
{\bf II.}
As a consequence, for every $T>0$, we have the convergence $z_\varrho\to z$ in $\C([0,T],\H)$ as $\varrho\to 0$, for some
vector $z=(u,\partial_t u,\partial_{tt} u)$. The next step is showing that
$u$ is a solution to the original Cauchy problem \eqref{volterra}-\eqref{IC}.
The desired initial conditions are trivially satisfied, so we are left to verify
that \eqref{volterra} holds true in the weak sense.
Since $u_\varrho$ solves the differential equation \eqref{volterra_rho}, the result follows by the convergence
$$\lim_{\varrho\to 0}\Q_{\varrho}(t) \l u_0, \zeta \r_1=0,$$
for any given $\zeta \in H^1$ and every $t\in[0,T]$, clearly implied by the bound \eqref{Qbound}.

\smallskip
\noindent
{\bf III.}
Exploiting the representation formula~\eqref{eta_repr}, we now write the energy
$\E_\varrho(t)$ of $u_\varrho$, defined in~\eqref{energyrho},
in the equivalent form
\begin{align*}
\E_\varrho(t)&=\|u_\varrho(t)\|^2_1+\|\partial_t u_\varrho(t)\|^2_1+\|\partial_{tt}u_\varrho(t)\|^2 \\
&\quad+ \int_{0}^{t}-g'(s)\|u_\varrho(t-s)-u_\varrho(t)\|_1^2 ds
+\int_{t}^{\infty}-g'(s)\|u_\varrho(t)+[\q_\varrho(s-t)-1]u_0\|_1^2 ds.
\end{align*}
The convergence $z_\varrho\to z$ in $\C([0,T],\H)$, for every $T>0$, entails that
$$\lim_{\varrho\to 0}\E_\varrho(t)=\E(t)+g(t)\|u(t)\|_1^2,\quad\forall t\geq 0,$$
where $\E(t)$ is the energy of the original problem \eqref{volterra}-\eqref{IC},
defined in~\eqref{energy}.
Passing to the limit
in the energy inequality \eqref{Ebdd}, we readily obtain~\eqref{BDD}
(recall that $g(0)<\infty$). This finishes the proof of Theorem~\ref{EU}.
By the same token, when~\eqref{gexp} holds,
the exponential decay estimate~\eqref{DECAYrho} for $\E_\varrho(t)$ of Theorem~\ref{EXPrho}
yields in the limit \eqref{DECAY} with
$$M=M_0[1+g(0)].$$
The proof of  Theorem~\ref{EXP} is done.
\qed
%%%%%%%%%%%%%%%%%%%%%%%%%%%%%%%%%%%%%%%%%%%%

%%%%%%%%%%%%%%%%%%%%%%%%%%%%%%%%%%%%%%%%%%%%

\bigskip
%%%%%%%%%%%%%%%%%%%%%%%%%%%%%%%%%%%%%%%%%%%%

%%%%%%%%%%%%%%%%%%%%%%%%%%%%%%%%%%%%%%%%%%%%
\end{document}